\documentclass[11pt]{amsart}
\usepackage{geometry}                
\usepackage{graphicx}
\usepackage{amssymb}
\usepackage{epstopdf}
\newtheorem{theorem}{Theorem}[section]
\newtheorem{lemma}{Lemma}[section]
\newtheorem{corollary}{Corollary}[section]

\newtheorem{conjecture}{Conjecture}[section]
\newtheorem{definition}{Definition}[section]

\numberwithin{equation}{section}

\def\Z{\mathbb Z}

\def\R{\mathbb R}

\def\D{\Delta}

\def\d{\partial}

\def\e{\epsilon}

\def\b{\beta}

\def\bfc{{\mathbf c}}
\def\bfc\Theta{\mathcal c\Theta}

\title[On bordisms of immersions and embeddings with trivial normal line bundles]{On immersions and embeddings with trivial normal line bundles}

\author[G.~Katz]{Gabriel Katz}
\address{MIT, Department of Mathematics, 77 Massachusetts Ave., Cambridge, MA 02139, U.S.A.}
                                           
\email {gabkatz@gmail.com}

\begin{document}
\maketitle 

\begin{abstract} Let $Z$ be a smooth compact $(n+1)$-manifold. We study smooth embeddings and immersions $\beta: M \to  Z$ of compact or closed $n$-manifolds $M$ such that the normal line bundle $\nu^\beta$ is trivialized. For a fixed $Z$, we introduce an equivalence relation between such $\beta$'s; it is a crossover between pseudo-isotopies and bordisms. We call this equivalence relation ``{\sf quasitopy}". It comes in two flavors: $\mathsf{IMM}(Z)$ and  $\mathsf{EMB}(Z)$, based on immersions and embeddings into $Z$, respectively. We prove that the natural map $\mathsf{A}:\mathsf{EMB}(Z) \to \mathsf{IMM}(Z)$ is injective and admits a right inverse $\mathsf{R}:\mathsf{IMM}(Z) \to \mathsf{EMB}(Z)$, induced by the resolution of self-intersections. As a result, we get a map $$\mathcal B\Sigma:\; \mathsf{IMM}(Z) \big/ \mathsf{A}(\mathsf{EMB}(Z)) \longrightarrow \bigoplus_{k \in [2, n+1]} \mathbf B_{n+1-k}(Z)$$ whose target is a collection of smooth bordism groups of the space $Z$ and which differentiate between immersions and embeddings.
\end{abstract}

\section{Introduction}

This paper contains few remarks that grew out of \cite{K1} and \cite{K2}. In the present paper, the subjects and settings of these two articles are drastically simplified and, therefore, merit a separate presentation. \smallskip

Let $Z$ be a smooth $(n+1)$-dimensional compact manifold with boundary, and $M$ a smooth compact $n$-dimensional manifold with boundary. We consider immersions/embedding $\b: (M, \d M) \to (Z, \d Z)$, whose normal $1$-bundle $\nu^\b =_{\mathsf{def}} \b^\ast(TZ) / TM$ is \emph{trivialized}. 

We denote by $\bar\b$ the immersion $\b$ in which the trivialization of $\nu^\b$ is flipped to the uposite.

\begin{definition}\label{def.IMM}
We say that two such immersions $\b_0$ and $\b_1$ are 
   {\sf cobordant} (or {\sf quasitopic} in the terminology of \cite{K1} and \cite{K2}), if there exists a smooth compact 
   $(n+1)$-dimensional manifold $N$, where $$\d N = (-M_0 \, \sqcup \, M_1)\bigcup_{\{-\d M_0\, \sqcup \, \d M_1\}} \delta N,$$ and an immersion $$B: (N, \delta N) \to (Z \times [0,1],\, \d Z \times [0,1])$$ so that: 
$B|_{M_0}  = \b_0$,  $B|_{M_1}  = \b_1$, and the trivialized normal bundle $$\nu^B =_{\mathsf{def}}  B^\ast(T(Z\times [0, 1])) / TN$$ of $B(N)$ in $Z \times [0,1]$ is an extension of trivialized bundles $\nu^{\b_0} \sqcup \nu^{\b_1}$. \smallskip

We denote by $\mathsf{IMM}(Z, \d Z)$ 
the set of bordism classes of such immersions. \smallskip

If we replace immersions $\b_0, \b_1, B$ in this setting with regular \emph{embeddings}, we get a similar notion of cobordism.  We denote by $\mathsf{EMB}(Z, \d Z)$ the set of embeddings, considered up to the cobordisms which are realized by embeddings in $Z \times [0,1]$ with trivialized normal bundles. 

We use the notations  $\mathsf{IMM}(Z)$,  $\mathsf{EMB}(Z)$ for similar constructions based on immersions/embeddings of \emph{closed} manifolds $M$. \hfill $\diamondsuit$

\end{definition}

Note that if $Z$ is orientable, then the triviality of the normal bundles $\nu^\b$ and of $\nu^B$ implies that $M$ and $N$ are orientable as well.\smallskip

For a given immersion $\b: (M, \d M) \to (Z,\, \d Z)$, we consider the $k$-fold product map $(\b)^k: (M)^k \to (Z)^k$, where $k \geq 2$. Let $\Sigma^\b_k$ be the preimage of the diagonal $\D \subset (Z)^k$ under the map $(\b)^k$. By definition, for a $k$-{\sf normal} immersion $\b$ (such $\b$'s are open and dense in the space of all immersions), $\Sigma^\b_k$ is a smooth manifold of dimension $n-k+1$ \cite{LS}. 

Following  \cite{LS}, we call $\Sigma^\b_k \subset (M)^k$ {\sf the $k$-self-intersection manifold of} $\b$. Note that the same construction applies to $\b|: \d M \to \d Z$, producing the boundary $\d\Sigma^\b_k$ of $\Sigma^\b_k$. \smallskip

Let $p_1: (M)^k \to M$ be the obvious projection on the first factor of the product. Then $p_1: \Sigma^\b_k \to M$ is an \emph{immersion} \cite{LS}. Its image $p_1(\Sigma^\b_k)$ is the set of points $x_1 \in M$ such that there exist distinct $x_2, \ldots , x_k \in M$ with the property $\b(x_2) = \b(x_1),\, \ldots ,\, \b(x_k) = \b(x_1)$. 
If both $M$ and $Z$ are orientable, then so is $\Sigma^\b_k$. 

This construction gives rise to a map $\b \circ p_1: (\Sigma^\b_k, \d\Sigma^\b_k)  \to (Z, \d Z)$. For $k \geq 2$, the map $\b \circ p_1$, regarded as an element of the bordism group $\mathbf B_{n-k+1}(Z, \d Z)$ or (when $Z$ is oriented) of the oriented bordism group $\mathbf{OB}_{n-k+1}(Z, \d Z)$, is an invariant that evidently distinguishes between immersions and embeddings. 

\section{Comparing bordisms of immersions and embeddings that have trivial normal line bundles}

Consider the obvious maps $$\mathsf{A}:\mathsf{EMB}(Z, \d Z) \to \mathsf{IMM}(Z, \d Z) \text{ \; and \; } \mathsf{A}:\mathsf{EMB}(Z) \to \mathsf{IMM}(Z).$$ 

\smallskip  

\begin{theorem}\label{th.6_inject} 
\begin{itemize}
\item The map $\mathsf{A}:\mathsf{EMB}(Z, \d Z) \to \mathsf{IMM}(Z, \d Z)$ 
is injective. In other words, if two embeddings, $\b_0$ and $\b_1$, with trivialized normal line bundles are cobordant by a cobordism that is an immersion with a trivialized normal line bundle, then they are cobordant by a cobordism that is an embedding with a trivialized normal line bundle.

\item There exists a surjective  resolution map $\mathsf{R}: \mathsf{IMM}(Z, \d Z) \to \mathsf{EMB}(Z, \d Z)$ such that the composition $\mathsf{R} \circ \mathsf{A} = \mathsf{id}$. Moreover, any $\b \in  \mathsf{IMM}(Z, \d Z)$ and $\mathsf{R}(\b) \in \mathsf{EMB}(Z, \d Z)$ define the same homology class in  $H_n(Z, \d Z; \Z)$, provided that $Z$ is orientable.\smallskip

\item Similar claims are valid for the map $ \mathsf{A}:\mathsf{EMB}(Z) \to \mathsf{IMM}(Z)$.
\end{itemize}
 \end{theorem}
 
 \begin{proof} Let two embedding $\b_0: (M_0, \d M_0) \to (Z, \d Z)$ and $\b_1: (M_1, \d M_1) \to (Z, \d Z)$ with trivialized normal bundles be connected by a cobordism 
 $$B: (N, \delta N) \to (Z \times [0,1], \d Z \times [0,1]),$$ an immersion with its normal $1$-bundle $\nu^B$ being trivialized so that the trivialization extends the trivializations of $\nu^{\b_0}$ and of $\nu^{\b_1}$. Then we can perturb $B$ so that it becomes $k$-{\sf normal} (see \cite{LS}) 
 for all $k \leq  n+2$ \cite{Th}.  The perturbation $B'$ of $B$ may be chosen so small that the normal $1$-bundle $\nu^{B'}$ is still trivial and  $B'|_{M_0 \sqcup M_1}= B|_{M_0 \sqcup M_1}$. 
 
 We aim to resolve $B'$ into an embedding $\tilde B$ of a new manifold $\tilde N$. Let us start with a local model of such a resolution. 
 
 \begin{figure}[ht]
\centerline{\includegraphics[height=2.5in,width=3.2in]{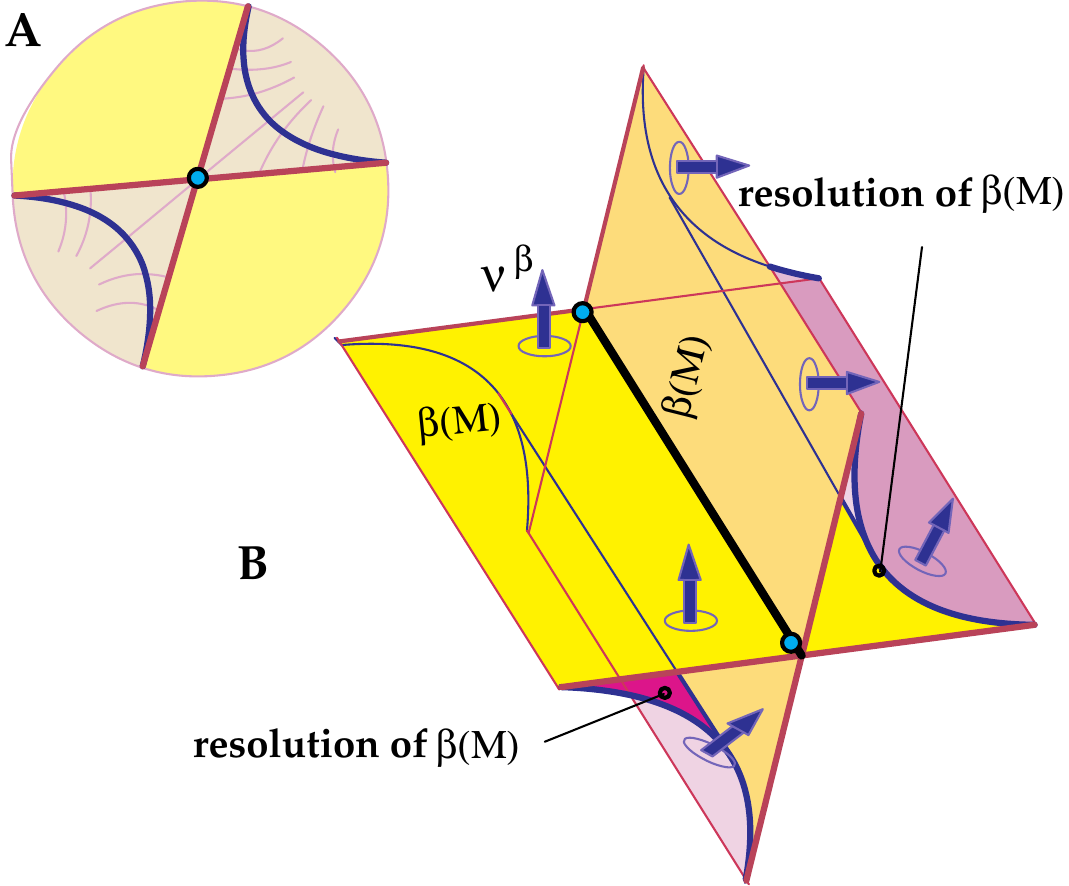}}
\bigskip
\caption{\small{A resolution of an immersion $\b: M \to Z $ into an embedding. Note that the orientation of the normal line bundle $\nu^\b$ helps to orient the normal line bundle of the resolution. The same picture may represent a  resolution of an immersion $B: N \to Z \times [0, 1]$ into an embedding.}}
\label{fig.resolution} 
\end{figure}

 Let $D^q$ be an open unit ball in $\R^q$, centered at the origin. We fix a  diffeomorphism $\phi: D^q \to \R^q$, where $\phi(\vec 0)= \vec 0$ and the radii of $D^q$ are mapped onto the infinite rays emanating from the origin. Let $x_1, \ldots , x_q$ be the canonical coordinates on $\R^q$. Consider the hypersurface $H_\e =_{\mathsf{def}} \{x_1 \cdot \ldots \cdot x_q = \e\}$, $\e > 0$. Replacing $H_0$ with $H_\e$,  or rather, $\tilde H_0 =_{\mathsf{def}} \phi^{-1}(H_0)$ with $\tilde H_\e =_{\mathsf{def}} \phi^{-1}(H_\e)$ will mimic the resolution we are after.   
 
 Put $B =_{\mathsf{def}} B'$. The resolution of $B$ is carried inductively: we assume that $B$  already has no self-intersections of multiplicities $ > q$ and consider the locus $Y =_{\mathsf{def}} Y_q^B \subset B(N)$ where exactly $q$ branches of $B(N)$ intersect transversally. Thus, $Y$ is  a compact manifold. Note that $\d Y \subset \d Z \times [0, 1]$ since $B|_{M_0 \sqcup M_1}$ is an embedding. 
 
We consider a tubular neighborhood $U$ of $Y$ in $Z \times [0, 1]$ such that $U \cap (\d Z \times [0, 1])$ is a tubular neighborhood of $\d Y$ in $\d Z \times [0, 1]$. 
 Let $\pi: U \to Y$ be the locally trivial fibration whose fibers are open $q$-balls. For a sufficiently narrow tube $U$, the intersection of $B(N)$ with $U$ is a fibration over $Y$ with the $\pi$-fiber $\tilde H_0$. 
 
 Over a small neighborhood $V_y$ of each point  $y \in Y$, $B(N)$ is given by a smooth map $\Psi: \pi^{-1}(V_y) \to B^q$, such that $\Psi$ is transversal to each branch of $\tilde H_0 \subset D^q$ and $\Psi^{-1}(\vec 0) = B(N) \cap \pi^{-1}(V_y)$. Thus, the structure group of the fibration $\pi: U \to Y$ is the group of diffeomorphisms $\mathsf{Diff}(D^q, \tilde H_0)$ of $D^q$ that fix $\vec 0$ and preserve $\tilde H_0$ invariant. In general, $\pi_0(\mathsf{Diff}(D^q, \tilde H_0)) \neq 1$. However, due to the global triviality of $\nu^B$, we have a consistent way of picking the preferred normals to the branches of $B(N) \cap U$.  In particular, we have a consistent way of picking the chambers in the fibers $\pi^{-1}(y) \setminus (\pi^{-1}(y) \cap B(N)) \approx D^q \setminus \tilde H_0$, where the product $x_1 \cdot \ldots \cdot x_q > 0$. We call such chambers {\sf positive}. Hence, the structure group of $\pi: U \to Y$ reduces to the subgroup $\mathsf{Diff}_+(D^q, \tilde H_0) \subset \mathsf{Diff}(D^q, \tilde H_0)$ whose elements map  negative chambers to negative chambers. For $\e > 0$, the group $\mathsf{Diff}_+(D^q, \tilde H_0)$  admits a further reduction to the subgroup $\mathsf{Diff}_+(D^q, \{\tilde H_\e\}_{\e \geq 0})$ whose diffeomorphisms preserve each hypersurface $\tilde H_\e$, $\e \geq 0$, invariant. Let us show that  $\mathsf{Diff}_+(D^q, \{\tilde H_\e\}_{\e \geq 0})$ is a deformation retract of  $\mathsf{Diff}_+(D^q, \tilde H_0)$. Indeed, consider the open set $D^q_+\subset D^q \setminus \tilde H_0$, the union of positive open chambers, and, for each $\e > 0$, its radial retraction $p_\e: D^q_+ \to \tilde H_\e$. For every $\psi \in \mathsf{Diff}_+(D^q, \tilde H_0)$, consider the map $\chi_\psi$, given by $\chi_\psi(a) =_{\mathsf{def}} p_{\e(a)} \circ \psi(a)$, where $a \in \bar D^q_+$ and $\e(a)$ is the non-negative value such that the hypersurface $\tilde H_{\e(a)}$ contains $a$. If $a \in \tilde H_0$, then put $\chi_\psi(a) =_{\mathsf{def}} \psi(a)$. Evidently, each $\chi_\psi \in \mathsf{Diff}_+(D^q, \{\tilde H_\e\}_{\e \geq 0})$. The $t$-family $\{p^t_{\e(a)} \circ \psi(a)\}_{t \in [0,1]}$ delivers the deformation of $\mathsf{Diff}_+(D^q, \tilde H_0)$ onto $\mathsf{Diff}_+(D^q, \{\tilde H_\e\}_{\e \geq 0})$.  Here $\{p^t_\e: D^q_+ \to D^q_+\}$ denotes the family of maps  that interpolate radially between $p_\e$ and the identity map $\mathsf{id}$.  
 
Therefore, we can reduce the structure group of the fibration $\pi: U \to Y$ to the subgroup $\mathsf{Diff}_+(D^q, \{\tilde H_\e\}_{\e \geq 0})$. As a result, the hypersurface $\tilde H_1$ (where $\e = 1$) in each $\pi$-fiber gives rise to a globally defined smooth hypersurface $\mathcal H \subset U$. 
Now, we resolve $B(N)$ into $$B^\dagger(N^\dagger) =_{\mathsf{def}}  \big[B(N) \setminus (B(N) \cap U)\big] \cup \mathcal H$$ with no self-intersections of the multiplicity $q$. Here $N^\dagger =_{\mathsf{def}} (N \setminus B^{-1}(U)) \cup_B \mathcal H$. 

Critically, this resolution is such that the trivialization of the normal $1$-bundle \hfill\break $\nu^B \big |_{\big[B(N) \setminus (B(N) \cap U)\big]}$ extends over $B^\dagger(N^\dagger)$: just consider the gradient of the product function $x_1 \cdot \ldots \cdot x_q$.
\smallskip 

Finally, the basis of induction by $q$ is the case ``$q = n+2$". For this case, $Y$ is a finite set, and the resolution $B^\dagger$ of $B$ works instantly.\smallskip 

Let us validate the second claim of the lemma. We also use an induction by $q$, whose base is now $q = n+1$. First, we observe that any immersion $\b: (M, \d M) \to (Z, \d Z)$ is cobordant in $\mathsf{IMM}(Z, \d Z)$ to an immersion $\tilde \b: (M, \d M) \to (Z, \d Z)$ which is $k$-normal for all $k \leq n+1$. Indeed, such $k$-normal immersions are dense in $C^\infty$-topology \cite{Th}. Any sufficiently $C^\infty$-close immersions are pseudo-isotopic and thus cobordant in $\mathsf{IMM}(Z, \d Z)$, the cobordism having a trivial normal bundle in $Z \times [0,1]$. 

Now the resolution map $\mathsf{R}: \mathsf{IMM}(Z, \d Z) \to \mathsf{EMB}(Z, \d Z)$ is defined by applying the previous resolution recipe ``$B^\dagger$" to $\tilde \b$ (and not to the cobordism $B$, as we just did).  Of course, the resolution $\b^\dagger = \mathsf{R}(\b)$ changes the topology of the manifold $M$. We denote by $M^\dagger$ the resulting manifold. 

Note that the $(n+1)$-chain $K \subset Z$, locally defined as $\bigcup_{\e \in [0,1]} \tilde H_\e$, has the boundary $-\tilde\b(M) \cup \b^\dagger(M^\dagger)$, which validates the equality $[\b(M)] = [\b^\dagger(M^\dagger)] \in H_n(Z, \d Z; \Z)$.
\smallskip

Similar claims for the injectivity of $\mathsf{A}:\mathsf{EMB}(Z) \to \mathsf{IMM}(Z)$ and its left inverse $\mathsf{R}: \mathsf{IMM}(Z) \to \mathsf{EMB}(Z)$ can be validated in exactly the same way. \hfill 
 \end{proof}
 %

Let $g$ be a Riemannian metric on $Z$, and $M$ a closed manifold. For any immersion $\b: M \to Z$ with the trivialized normal line bundle $\nu^\b$, we denote by $\b^\ast(g)$ the pull-back of $g$ to $M$. Let $vol(\b, g)$ be the volume of $M$ in the metric $\b^\ast(g)$.

\begin{corollary}\label{cor.6_minimal_hypersurface} Let $M$ be a closed $n$-manifold. If an immersion $\b: M \to Z$ with the trivialized normal line bundle $\nu^\b$ minimizes the volume $vol(\b, g)$ in the homology  class $[\b(M)] \in H_n(Z; \Z)$, then $\b$ is a regular embedding.
\end{corollary}

\begin{proof} Again, we argue by an induction in $q$,  whose base is  $q = n+1$.
We assume that $\b$  already has no self-intersections of multiplicities $ > q$ and consider the locus $Y_q$ where exactly $q$ branches of $\b(M)$ intersect transversally. As before, we pick a sufficiently narrow tubular neighborhood $U$ of $Y_q$ in $Z$ in which the resolution $\b^\dagger$ of the minimal $\b$ is supported. 

We claim that $vol(\b, g) > vol(\b^\dagger, g)$ for a sufficiently narrow $U$. Indeed, outside $U$, $\b(M) = \b^\dagger(M^\dagger)$; thus it is sufficient to compare the $g$-induced volumes of $\b(M) \cap U$ and  $\b^\dagger(M^\dagger) \cap U$.  We may assume that $U$ has the $(n+1-q)$-disk bundle structure over a $q$-manifold $Y$, where $Y$ is the locus of $(n+1-q)$-fold self-intersection $\b(M)$, and the disk fibers $\{U_y\}$ of the bundle $U \to Y$ are orthogonal to $Y$ in the metric $g$. By the Fubini Theorem, 
$$vol_g(\b(M) \cap U) = \int_Y vol_g(\b(M) \cap U_y)\, dg|_Y, \text{ and } $$ 
$$vol_g(\b^\dagger (M^\dagger) \cap U) = \int_Y vol_g(\b^\dagger(M^\dagger) \cap U_y)\, dg|_Y.$$
 However, for each $y \in Y$, (thanks to the hyperbolic-like geometry of the resolution $\b^\dagger(M^\dagger) \cap U_y$ and the near-euclidean geometry of the fibers $U_y$, as depicted in Fig. \ref{fig.resolution}, A) 
 the gradient flow of the product function $\phi^\ast(x_1 \cdot \ldots \cdot x_q)$, in the positive chambers, decreases the $n$-volume of  compact domains in the constant level loci $\tilde H_\e$ of $\phi^\ast(x_1 \cdot \ldots \cdot x_q)$. 
 Therefore,
 $$vol_g(\b(M) \cap U_y) > vol_g(\b^\dagger(M^\dagger) \cap U_y).$$  By Theorem \ref{th.6_inject}, $[\b(M)] = [\b^\dagger(M^\dagger)]$. Thus, the inequality $vol(\b, g) > vol(\b^\dagger, g)$  contradicts the hypotheses that $\b$ is the volume-minimizing immersion in its homology class $[\b(M)] \in H_n(Z; \Z)$.  Therefore, by induction, the the volume-minimizing $\b$ cannot have self-intersections. So, it must be an embedding.
\end{proof}

\noindent {\bf Remark 2.1.} 
It is well-known that any integral homology class $h \in H_n(Z; \Z)$ is realizable as the image of the fundamental class  $[M]$ under an embedding $\b: M \hookrightarrow Z$ of some closed smooth $n$-manifold $M$ (see the proof of Lemma \ref{lem.EMB_is_H}). (In general, when $\dim Z - \dim M \geq 2$, even not any $\Z_2$-homology class has a realization by an immersion $\b: M \to Z$!)\smallskip

We call a Riemannian metric $g$ on a compact $(n+1)$-manifold $Z$ $\mathsf n$-{\sf ample}, if any homology class $h \in H_n(Z; \Z)$ is realizable by a {\sf minimal} hypersurface with trivial normal line bundle, i.e., by an \emph{embedding} $\b: M  \hookrightarrow  Z$ with trivial line bundle $\nu^\b$ and with the minimal volume $vol(\b, g)$ among all embeddings $\b$, subject to $\b([M])= h$. For example, the flat metric on a torus $T^{n+1}$ is $n$-ample, and so is the hyperbolic metric on any closed orientable surface of genus $\geq 2$.
%
\smallskip

For a $n$-ample metric $g$, Corollary \ref{cor.6_minimal_hypersurface} implies that $vol(\b, g)$ gives rise to a {\sf norm} $\| h\|_g$ on the space $H_n(Z; \R)$, similar to the Thurston norm for $3$-folds \cite{Thur}. 

Note that, for an $n$-ample $g$, it is sufficient to define the ``norm" $$\| h \|_g =_{\mathsf{def}} \inf_{\{\text{embeddings } \b:\, M \hookrightarrow Z, \;\; \b([M])\, =\, h \}} vol(\b, g)$$ on the integral lattice $H_n(Z; \Z) \subset H_n(Z; \R)$ so that the triangle inequality $\| h_1 + h_2 \|_g \leq \| h_1 \|_g  + \| h_2 \|_g$ holds for all vectors $h_1, h_2$ of the lattice.

Let us validate the triangle inequality on the lattice $H_n(Z; \Z)$.
Take any pair of $n$-volume minimizing embeddings $\b_1: M_1  \hookrightarrow  Z$ and $\b_2: M_2  \hookrightarrow  Z$, such that $\b_1([M_1]) = h_1$ and $\b_2([M_2]) = h_2$. Then the immersion $\b_1 \sqcup \b_2: M_1 \sqcup M_2 \to Z$,  which realizes the homology class $h_1+ h_2$, by the resolution construction $\mathsf R$ of Corollary \ref{cor.6_minimal_hypersurface}, has the property $$\| h_1+ h_2\|_g \leq vol(\mathsf R(\b_1 \sqcup \b_2), g) \leq vol(\b_1 \sqcup \b_2, g) = vol(\b_1, g) + vol(\b_2, g) = \| h_1\|_g + \| h_2\|_g .$$ 

We wonder whether, for an $n$-ample $g$, the unit ball $\{h \in H_n(Z; \R) | \; \| h\|_g \leq 1\}$ is a convex polyhedron?
\hfill $\diamondsuit$
 
\begin{lemma}\label{lem.EMB_is_H} There are bijective maps $\mathcal G: \mathsf{EMB}(Z) \stackrel{\approx}{\rightarrow}  H^1(Z, \d Z; \Z)$ and $\mathcal G: \mathsf{EMB}(Z, \d Z) \stackrel{\approx}{\rightarrow}  H^1(Z; \Z)$. As a result, $\mathsf{EMB}(Z)$ depends only on the homology type of $(Z, \d Z)$,  and \hfill\break $\mathsf{EMB}(Z, \d Z)$ depends only on the homology type of $Z$.
\end{lemma}
 
 \begin{proof} Let $M$ be a closed manifold. If $\b: M \to Z$ is an immersion which is an embedding whose normal bundle $\nu^\b$ is trivialized, then applying the Thom construction to $\nu^\b$, we get a map $G(\b): (Z/\d Z) \to Th(\nu^\b) \to Th(\mathbf{1}) \approx S^1$. Let $\star$ be the base point in $S^1$. Note that $G(\b)(\d Z) = \star$. Using Definition \ref{def.IMM}, we conclude that cobordant regular embeddings $\b_0$ and $\b_1)$ lead to the homotopic maps $G(\b_0)$ and $G(\b_1)$.  Therefore, the map $\mathcal G: \mathsf{EMB}(Z) \to [(Z, \d Z), (S^1, \star)]$, induced by $G(\b)$, is well-defined.
 
 On the other hand, if a map $G: (Z, \d Z) \to (S^1, \star)$ is transversal to a point $a \in S^1$, $a \neq \star$, then the preimage $G^{-1}(a)$ (a closed submanifold) delivers a regular embedding $\b_G: G^{-1}(a) \subset Z$ so that $G(\b_G) = G$. Thus, $\mathcal G$ is surjective. 
 
 If $G(\b): (Z, \d Z) \to (S^1, \star)$ is nill-homotopic via a homotopy $H : (Z \times [0, 1], \d Z \times [0, 1]) \to (S^1, \star)$, transversal to $a$, then $(G(\b))^{-1}(a) \subset Z$ is cobordant to the empty embedding via the cobordism  $H^{-1}(a)$. 
 
A similar argument, being applied to a regular embedding $\b: (M, \d M) \to (Z, \d Z)$, produces a map  $G(\b): Z \to Th(\nu^\b) \to Th(\mathbf{1}) \approx S^1$ and, therefore, an element of $H^1(Z; \Z)$, which depends only on the bordism class of $\b$.  The fact that $\mathcal G: \mathsf{EMB}(Z, \d Z) \to H^1(Z; \Z)$ is a bijection has a similar validation.
 \end{proof}
 
 \begin{conjecture} The sets  $\mathsf{IMM}(Z, \d Z)$ and $\mathsf{IMM}(Z)$ depend only on the homotopy type of $Z$ and of $(Z, \d Z)$, respectively. \hfill $\diamondsuit$
 \end{conjecture}
 
\noindent {\bf Remark 2.2.}
Note that flipping the orientation of the normal bundle $\nu^\b$ leads to the identity $G(\b) = - G(\bar\b)$ in $H^1(Z, \d Z; \Z)$. Thus, when $H^1(Z, \d Z; \Z) = 0$ or $H^1(Z, \d Z; \Z) = \Z_2$, inverting a ``hairy $n$-cycle" $\b(M) \subset Z$, viewed as an element of $\mathsf{EMB}(Z)$, inside $Z \times [0, 1]$ is easy. Of course, to invert $\b$, via a path in the space of immersions, is much harder \cite{S}, \cite{Hi}. The possibility of such an inversion ``$\b = \bar\b$" for the standard embedding $\b: S^2 \to D^3$ of the hairy sphere follows from the striking result of Smale \cite{S}. Its visual realization is very intricate indeed. 
\hfill $\diamondsuit$
\smallskip

Following the arguments in \cite{K1},  Theorem \ref{th.6_inject} leads directly to Corollary \ref{cor.IMM/EMB} below, 
 where $\mathbf B_j(\sim)$ stands for the $j$-dimensional smooth bordisms and $\mathbf{OB}_j(\sim)$ for the $j$-dimensional oriented smooth bordisms of the relevant space \cite{St}. Recall that if $Z$ is oriented, then each $M$ acquires an orientation with the help of the oriented line bundle $\nu^\b$.
 
\begin{corollary}\label{cor.IMM/EMB} The relative bordism classes of the $k$-fold normal self-intersection manifolds  $\{[\b\circ p_1: \Sigma_k^\b \to (Z, \d Z)]\}_{2 \leq k \leq n+1}$ of immersions $\b: (M, \d M) \to (Z, \d Z)$ with trivial normal bundle $\nu^\b$ give rise to a map 
\begin{eqnarray}\label{eq.IMM_to_Sigma} 
\mathcal B\Sigma:\; \mathsf{IMM}(Z, \d Z) \longrightarrow \bigoplus_{k \in [2, n+1]} \mathbf B_{n+1-k}(Z, \d Z) 
\end{eqnarray}
whose ``kernel" $\mathcal B\Sigma^{-1}(\mathbf 0)$ contains $\mathsf{EMB}(Z, \d Z) \approx H^1(Z; \Z)$. \smallskip

The bordism classes of the $k$-fold normal self-intersection manifolds  $\{[\b\circ p_1: \Sigma_k^\b \to Z]\}_{2 \leq k \leq n+1}$ of immersions $\b: M \to Z$ ($M$ being closed) with trivial normal bundle $\nu^\b$ give rise to a map
\begin{eqnarray}\label{eq.IMM_to_Sigma_A} 
\mathcal B\Sigma:\; \mathsf{IMM}(Z) \longrightarrow \bigoplus_{k \in [2, n+1]} \mathbf B_{n+1-k}(Z) 
\end{eqnarray}
whose ``kernel" $\mathcal B\Sigma^{-1}(\mathbf 0)$ contains $\mathsf{EMB}(Z) \approx H^1(Z, \d Z; \Z)$. \smallskip
 
If $Z$ is oriented, then the target of $\mathcal B\Sigma$ is $\bigoplus_{k \in [2, n+1]} \mathbf{OB}_{n+1-k}(Z, \d Z)$, the relative oriented bordisms, or $\bigoplus_{k \in [2, n+1]} \mathbf{OB}_{n+1-k}(Z)$, the oriented bordisms. 
  \hfill $\diamondsuit$ 
\end{corollary}

Here is the simplest example of an invariant $$\rho(\b) =_{\mathsf{def}} 
\big[(\b \circ p_1): \Sigma_{n+1}^\b \to Z\big ],$$ delivered by $\mathcal{B}\Sigma$ from (\ref{eq.IMM_to_Sigma_A}), which does discriminate between the quasitopies of embeddings and immersions. By the definition of $\Sigma_{n+1}^\b$, the number of points in $\Sigma_{n+1}^\b$ is divisible by $n+1$.  Let $\rho_\b = \frac{1}{n+1} \#(\Sigma_{n+1}^\b)$ be the number of points in $Z$ where exactly $n+1$ branches of $\b(M)$ meet transversally. If $Z$ is oriented, then $\rho_\b \in \Z$; otherwise, $\rho_\b \in \Z_2$.
 
\begin{lemma}\label{odd_intersections}
Let $\dim M = n$ and let $\rho_\b$ be the number of points in $Z$ where exactly $n+1$ branches of $\b(M)$ meet transversally. If $\rho_\b \equiv 1 \mod 2$, then $\b \in \mathsf{IMM}(Z)\big/ \mathsf A(\mathsf{IMM}(Z))$ is nontrivial. 
\end{lemma}

 \begin{proof}
 Let $F_\b$ denote the finite set of points in $Z$, where exactly $n+1$ branches of $\b(M)$ meet transversally in $Z$.  Assume that $\b$ be quasitopic to the trivial immersion with the help of a cobordism $B: N \to Z \times [0, 1]$. Then, in general, $B(N)$ may have points in $Z \times [0, 1]$ where exactly $n+2$ branches of $B(N)$ meet transversally. Let $E_B$ denote their finite set and let $\eta_B =_{\mathsf{def}} \#(E_B)$. The points, where exactly $n+1$ branches of $B(N)$ meet transversally, form a graph $\Gamma_B \subset  Z \times [0, 1]$. The graph $\Gamma_B$ has $\rho_\b$ univalent vertices and $\eta_B$ vertices of valency $n+2$. Every edge of $\Gamma_B$ that does not terminate at a valency one vertex from $F_\b$ is attached to $E_B$. Thus counting the edges of $\Gamma_B$ we get $2(n+2)\eta_B - \rho_\b \equiv 0 \mod 2$. Therefore, when $\rho_\b \equiv 1 \mod 2$,  we get a contradiction with the assumption that $\b$ is quasitopic to the trivial immersion. As a result, any immersion $\b$ with an odd number $\rho_\b$ is nontrivial in $\mathsf{IMM}(Z)\big/ \mathsf A(\mathsf{IMM}(Z))$.
\end{proof}

\noindent {\bf Example 2.1.}
Let $Z$ be the $3$-torus or the $3$-torus from which an open $3$-ball $B$ is deleted. Let $M_\star = T^2_1 \coprod T^2_2 \coprod  T^2_3$ be the disjoint union of three $2$-tori, and let $\b_\star: M_\star \to Z$ be the immersion of $M_\star$, whose image is the triple of coordinate subtori in $Z$ which miss $B$ and have a single triple self-intersection. Thus, $\rho_{\b_\star} =1$. 

Since the torus $Z$ is oriented and $H^1(Z, \d Z; \Z) \approx \Z^3$, by Corollary \ref{cor.IMM/EMB}, we get a sequence of maps: 
$$0 \to \Z^3 \stackrel{\mathsf A}{\rightarrow} \mathsf{IMM}(Z) \stackrel{\tilde{\mathcal B}\Sigma}{\rightarrow} \Z^3  \to 0,$$ 
where $\mathsf A$ is injective, $\tilde{\mathcal B}\Sigma$ is surjective, and their composition is trivial. The map $\tilde{\mathcal B}\Sigma$ is a truncation of the map $\mathcal B\Sigma$ from (\ref{eq.IMM_to_Sigma_A}), or rather of its oriented version.  For  (co)homological reasons, in 
the target of $\mathcal B\Sigma$, the oriented $1$-bordism classes  $\{\b \circ p_1: \Sigma_2^\b \to Z\}$ of double intersections determine the oriented $0$-bordism class of triple intersections $\b \circ p_1: \Sigma_3^\b \to Z$. Therefore, we dropped the $0$-bordism component of  $\mathcal B\Sigma$ and formed its truncation $\tilde{\mathcal B}\Sigma$.
\hfill $\diamondsuit$
 \smallskip


Given two immersions $\b_1: M_1 \to Z_1$ and $\b_2: M_2 \to Z_2$, where $M_1, M_2$ are closed $n$-manifolds and $Z_1$ and $Z_2$ are compact connected $(n+1)$-manifolds with connected boundaries, we can form a boundary connected sum $Z_1 \#_\d Z_2$. Consider the obvious map $\b_1 \sqcup \b_2: M_1 \sqcup M_2 \to Z_1 \#_\d Z_2$ that leaves the $1$-handle in $Z_1 \#_\d Z_2$ untouched. This construction gives rise to a coupling 
\begin{eqnarray}\label{eq.IMM_coupling}
\uplus: \mathsf{IMM}(Z_1) \times \mathsf{IMM}(Z_2) \to \mathsf{IMM}(Z_1 \#_\d Z_2).\end{eqnarray}

Let $D^{n+1}$ be a smooth $(n+1)$-dimensional ball. Then, for any connected $(n+1)$-manifold $Z$ with a connected boundary,  we get the coupling
\begin{eqnarray}\label{eq.IMM_coupling_A}
\uplus: \mathsf{IMM}(D^{n+1}) \times \mathsf{IMM}(Z) \to \mathsf{IMM}(Z)
\end{eqnarray}
since $D^{n+1} \#_\d Z \approx Z$.
\smallskip

Arguing as in \cite{K1}, Proposition 3.2, we check that both sets, $\mathsf{IMM}(D^{n+1})$ and $\mathsf{EMB}(D^{n+1})$, are {\sf groups}, where the group operation is induced by a connected sum operation $\uplus$ in (\ref{eq.IMM_coupling}). 

\begin{figure}[ht]
\centerline{\includegraphics[height=3in,width=4.7in]{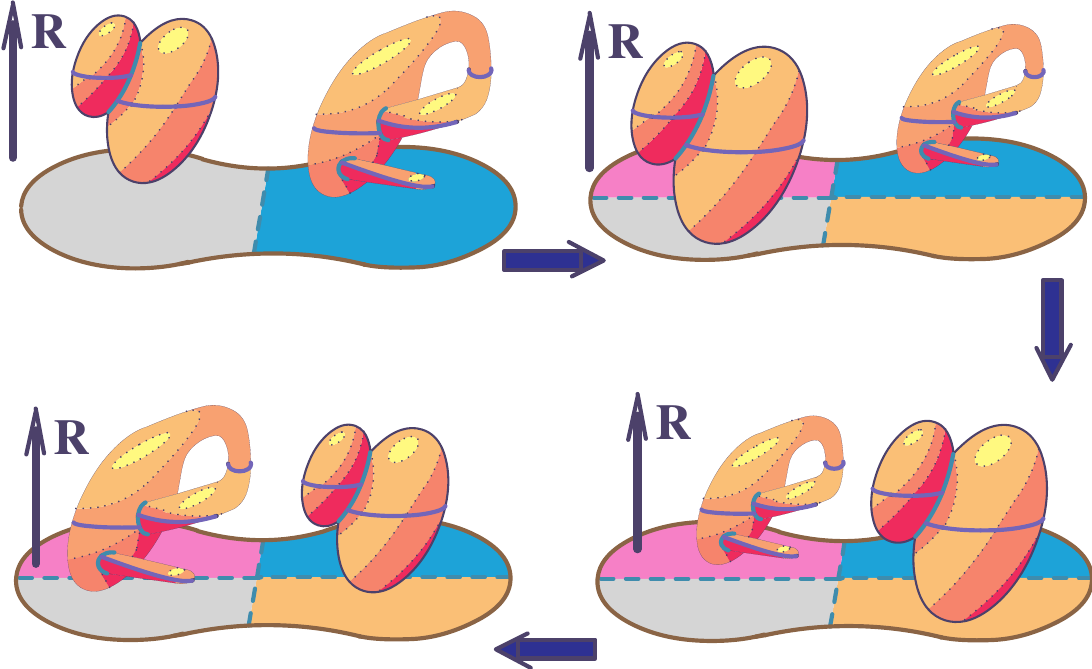}}
\bigskip
\caption{\small{Proving the commutativity of the operation $\uplus$ in the group $\mathsf{IMM}(D^{n+1})$ for $n \geq 2$. The three bold arrows show the effects of isotopies in the ball $D^{n+1}$.}} 
\label{fig.commutativity}
\end{figure}

\begin{theorem} 
\begin{itemize}
\item For $n > 0$, the group $\mathsf{EMB}(D^{n+1})$ is trivial,
\smallskip

\item For $n > 0$, the group $\mathsf{IMM}(D^{n+1})$ is abelian. 

\item For any compact connected smooth $(n+1)$-manifold $Z$ with a connected boundary, $\mathsf{IMM}(D^{n+1})$ acts on the set $\mathsf{IMM}(Z)$ via the connected sum operation $\uplus$ as in (\ref{eq.IMM_coupling}).\smallskip

\item The surjective resolution map $\mathsf{R}:\mathsf{IMM}(Z) \to \mathsf{EMB}(Z) \approx H^1(Z, \d Z; \Z)$ is equivariant with respect to this action. 
\end{itemize}
\end{theorem}

\begin{proof} By the arguments as in \cite{K1}, Proposition 3.2, $\mathsf{IMM}(D^{n+1})$ and $\mathsf{EMB}(D^{n+1})$ are abelian groups for $n\geq  1$. Moreover, they act on the set $\mathsf{IMM}(Z)$ via the connected sum operation $\uplus$ as in (\ref{eq.IMM_coupling_A}).

By Lemma \ref{lem.EMB_is_H}, $\mathsf{EMB}(D^{n+1}) \approx H^1(D^{n+1}, \d D^{n+1}; \Z) = 0$ for $n \geq 1$.

Tracing the arguments from the proof of Theorem \ref{th.6_inject}, the surjective resolution map $\mathsf{R}:\mathsf{IMM}(Z) \to \mathsf{EMB}(Z) \approx H^1(Z, \d Z; \Z)$ is equivariant with respect to $\mathsf{IMM}(D^{n+1})$-action.
\end{proof}



\end{document}